\documentclass{article}

\usepackage{epsfig}
\usepackage{latexsym}

\usepackage{epsfig}
\usepackage{latexsym}
\usepackage{amsmath,amssymb}
\input{epsf}

\newtheorem{prelem}{{\bf Theorem}}

\newtheorem{theorem}{Theorem}[section]

\newtheorem{lemma}[theorem]{Lemma}
\newtheorem{proposition}[theorem]{Proposition}

\newenvironment{proof}{{\bf Proof.}}{\hfill\rule{2mm}{2mm}}
\newtheorem{remarka}[theorem]{Remark}

\def\pr {{\rm \bf Pr}}

\def\Ex {\mathbb{E}}
\def\ex {\mathbb{E}}

\def\Var {{\rm \bf Var}}

\def\D{\Delta}
\def\e{\epsilon}

\def\g{\gamma}

\def\k{\kappa}

\def\z{\zeta}

\def\la{\lambda}

\def\r{\rho}

\def\s{\sigma}

\def\t{\tau}

\def\cald{{\cal D}}
\def\calh{{\cal H}}
\def\calp{{\cal P}}

\def\hf{{1\over2}}
\def\cm{{\cal C}_{\rm max}}

\newcommand{\proofend}{\hspace*{\fill}\mbox{$\Box$}}
\newcommand{\limninf}{\lim_{n \rightarrow \infty}}

\newcommand{\inv}[1]{\mbox{$1\over #1 $}}

\title{The scaling window for a random graph with a given degree sequence}
\author{
{\bf  Hamed Hatami and Michael Molloy\footnote{Research supported by an NSERC Discovery Grant.}} \\
{\small\it Department of Computer Science}\\
{\small University of Toronto} \\
{\small e-mail: hamed@cs.toronto.edu, molloy@cs.toronto.edu}}
\date{}
\begin{document}
\maketitle

%\tableofcontents

\begin{abstract} We consider a random graph on a given degree sequence $\cald$,
satisfying certain conditions.  We focus on two parameters $Q=Q(\cald), R=R(\cald)$.
Molloy and Reed proved that $Q=0$ is the threshold for the random graph to
have a giant component.  We prove that if $|Q|=O(n^{-1/3} R^{2/3})$ then, with high probability,
the size of the largest component of the random graph will be of order $\Theta(n^{2/3}R^{-1/3})$.
If $|Q|$ is asymptotically larger than $n^{-1/3}R^{2/3}$ then the size of the largest component is
asymptotically smaller or larger than $n^{2/3}R^{-1/3}$. Thus, we establish that the scaling window
is $|Q|=O(n^{-1/3} R^{2/3})$.
\end{abstract}

%%%%%%%%%%%%%%%%%%%%%%%%%%%%%%%%%%%%%%%%%%%%%%%%%%%%%%%%%%%%%%%%%%%%%
\section{Introduction}

The double-jump threshold, discovered by Erd\H{o}s and R\'enyi\cite{er}, is
one of the most fundamental phenomena in the theory of random graphs.  The component
structure of the random
graph $G_{n,p=c/n}$ changes suddenly when $c$ moves from below one to above one.
For every constant $c<1$,  almost surely\footnote{A
property P holds {\em almost surely} if $\limninf\pr(P)=1$.}
 (a.s.) every component has size $O(\log n)$, at $c=1$ a.s. the largest
component has size of order $\Theta(n^{2/3})$, and at $c>1$ a.s. there exists a single giant component
of size $\Theta(n)$ and all other components have size $O(\log n)$.  For this reason, $c=1$ is often
referred to as the {\em critical point}.

In the 1980's, Bollob\'as\cite{bb2}, {\L}uczak\cite{tl} and others studied the case
where $p=\frac{1+o(1)}{n}$.  They showed that when $p=\inv{n}+\frac{c}{n^{1/3}}$ for
any constant $c$ (positive or negative), the component sizes of $G_{n,p}$ behave as described above for $p=\inv{n}$. Furthermore, if $p$ lies outside of that range, then the size of the largest component behaves very differently: For larger/smaller values of $p$, a.s. the largest component has size asymptotically larger/smaller than $\Theta(n^{2/3})$. That range of $p$ is generally referred to as the {\em scaling window}. See, eg. \cite{bbrg} for further details.

Molloy and Reed\cite{mr1} proved that something analogous to the cases $c<1$ and $c>1$
holds for random graphs on a given degree sequence.  They considered a sequence
$\cald=(d_1,...,d_n)$ satisfying certain conditions,
and chose a graph uniformly at random from amongst all graphs with that degree sequence.
They determined a parameter $Q=Q(\cald)$ such that if $Q<0$ then a.s.  every
component has size $O(n^x)$ for some $x<1$ and if $Q>0$ then a.s. there exists a single giant component
of size $\Theta(n)$ and all other components have size $O(\log n)$.

In this paper, we establish a scaling window around the threshold $Q=0$, under certain conditions for $\cald$.
We will state our results more formally in the next subsection, but in short:
If $\sum d_i^3=O(n)$, then the situation is  very much like that for $G_{n,p}$.  The scaling
window is the range $|Q|=O(n^{-1/3})$ and inside the scaling window, the size of the largest
component is $\Theta(n^{2/3})$. As discussed below, the conditions required in \cite{ks,jl} imply that
$\sum d_i^3=O(n)$, which explains why they obtained their results.
If $\sum d_i^3 \gg n$, then the situation changes: the size of the scaling window becomes asympotically larger,
and the size of the largest component becomes asymptotically smaller.

\subsection{The main results}
Before stating our theorems, we will introduce some notation:

We are given a set of vertices along with the degree $d_v$ of each vertex.
We denote this degree sequence by ${\cal D}$.  We assume that
there is at least one graph with degree sequence $\cald$ (and so, eg., $\sum_v d_v$ is even).
Our random graph is selected uniformly from amongst all graphs with degree sequence $\cald$.

We use $E$ to denote the set of edges, and note that $|E|=\hf\sum_{v\in G}d_v$. We let $n_i$ denote the
number of vertices of degree $i$.
We use ${\cal C}_{\rm max}$ to denote the largest component of the random graph.
We define:
$$Q:=Q({\cal D}):= \frac{\sum_{u \in G} d_u^2}{2|E|}-2,$$
$$R:=R({\cal D}):= \frac{\sum_{u \in G} d_u(d_u-2)^2}{2|E|}.$$
The relevance of $Q,R$ will be made clear in Section \ref{s24}.
The asymptotic order of $R$ is important; note that, when $|E|/n$ and $Q$ are bounded by constants,
$R$ has the same order as $\inv{n}\sum_{u\in G}d_u^3$. The order of $R$ was implicitly seen to
be important in the related papers \cite{jl,ks}, where they required $\inv{n}\sum_{u\in G}d_u^3$
to be bounded by a constant (see Section \ref{srw}).

Molloy and Reed \cite{mr1} proved that, under certain assumptions about $\cald$,
if $Q$ is at least a positive constant, then a.s. $|\cm|\geq cn$ for some $c>0$ and if $Q$ is at most a negative constant then
a.s.  $|\cm|\leq n^x$ for some constant $x<1$.
Some of these assumptions were that the degree sequence converged in certain ways
as $n\rightarrow\infty$; in particular, $n_i/n$ converged to a limit for all $i$ uniformly,
and $Q$ converged to $\frac{\sum_{i\geq0} i^2\times \limninf n_i/n}{2|E|}-2$.
We don't require those assumptions in this paper.

But we do require some assumptions about our degree sequence.
First, it will be convenient to assume that every vertex has degree at least one.
A random graph with degree sequence $d_1,...,d_n$ where $d_i=0$ for every $i>n'$ has the same distribution
as a random graph with degree sequence $d_1,...,d_{n'}$ with $n-n'$ vertices of degree zero added to it.
So it is straightforward to apply our results to degree sequences with
vertices of degree zero.

Anomalies can arise when $n_2=n-o(n)$.  For example, in the extreme case where $n_2=n$,
we have a random 2-regular graph, and in this case the largest component is known
to have size $\Theta(n)$ (see eg. \cite{abt}).
So we require that $n_2\leq (1-\z)n$ for some constant $\z>0$. \cite{jl,ks} required that
$n_1>\z n$ - note that requirement is equivalent to ours when $n_0=0$ and $Q=o(1)$.
See Remark 2.7 of \cite{jl} for a description of some other  behaviours that can
arise when we allow $n_2=n-o(n)$.
%Again, our results extend easily to degree sequences with an unlimited
%number of degree two vertices;
%by noting that we can form such a random graph by: (i) choosing
%a uniformly random graph on $n-n_2$ vertices which has $n_i$ vertices of degree $i$ for each $i\neq 2$;
%(ii) adding some appropriately selected components that each consist of a single cycle;
%(iii) inserting the remaining degree two vertices onto some randomly selected edges.
%see Corollary \ref{c2} and its proof for more details.

As in \cite{mr1,mr2} and most related papers (eg. \cite{fr,jl,ks}),
we require an upper bound on the maximum degree, $\D$.  We take $\D\leq n^{1/3}R^{1/3}(\ln n)^{-1}$,
which is higher than the bounds from \cite{jl,ks,mr1,mr2} and is nearly as high as $\D$
can possibly be in this setting (see Section \ref{sd}).

Finally, since we are concerned with $Q=o(1)$, we can assume $|Q|\leq\frac{\z}{2}$,
and that $\z$ is sufficiently small, eg. $\z<\inv{10}$.
In summary, we assume that $\cald$ satisfies the following:

\noindent{\bf Condition D: }{\em For some constant $0<\z<\inv{10}$
 \begin{enumerate}
\item[(a)] $\Delta \le n^{1/3}R^{1/3}(\ln n)^{-1} $;
\item[(b)] $n_0=0$;
\item[(c)] $n_2\leq(1-\z)n$;
\item[(d)] $|Q|\leq\frac{\z}{2}$.
\end{enumerate}}

Our main theorems are:

\begin{theorem}
\label{thm:UpperBound} For any $\la,\e,\z>0$ there exist $A,B$ and $N$
such that for any $n\geq N$ and any degree sequence $\cald$
satisfying Condition D and with  $-\lambda n^{-1/3} R^{2/3}\leq Q \le \lambda n^{-1/3} R^{2/3}$,
we have
\begin{enumerate}
\item[(a)]$\Pr[|{\cal C}_{\rm max}| \le A n^{2/3} R^{-1/3} ] \le \epsilon$;
\item[(b)]$\Pr[|{\cal C}_{\rm max}| \ge B n^{2/3} R^{-1/3}  ] \le \epsilon$.
\end{enumerate}
\end{theorem}

\begin{theorem}\label{t2} For any $\e,\z>0$ and any function $\omega(n)$ tending to $\infty$ with $n$,
there exists $B,N$ such that
for  any $n\geq N$ and any degree sequence $\cald$
satisfying Condition D and with $Q<-\omega(n)n^{-1/3} R^{2/3}$ we have:
\begin{enumerate}
\item[(a)] $\pr(|\cm|\geq B\sqrt{n/|Q|})<\e$.
\item[(b)] The probability that the random graph contains a component with more than one cycle is at most
$\frac{20}{\omega(n)^3}$.
\end{enumerate}
\end{theorem}

\begin{theorem}\label{t3} For any $\e,\z>0$ and any function $\omega(n)$ tending to $\infty$ with $n$,
there exists $A,N$ such that
for  any $n\geq N$ and any degree sequence $\cald$
satisfying Condition D and with $Q>\omega(n)n^{-1/3} R^{2/3}$:
\[\pr(|\cm|\leq AQn/R)<\e.\]
\end{theorem}

Note that the bounds on $|\cm|$ in Theorems \ref{t2} and \ref{t3} are
$B\sqrt{n/|Q|}<Bn^{-1/3}R^{2/3}/\sqrt{\omega(n)}$ and $AQn/R>A\omega(n)n^{2/3}R^{-1/3}$. So
our theorems imply that $|Q|=O(n^{-1/3} R^{2/3})$ is the scaling window for any degree sequences
that satisfy Condition D, and that in the scaling window
the size of the largest component is $\Theta(n^{2/3} R^{-1/3})$.

Note also that Theorem \ref{t2}(b) establishes that when $Q$ is below the scaling window
then, with high probability, every component is either a tree or is unicyclic.  This was
previously known to be the case for the $G_{n,p}$ model\cite{tl}.

The approach we take for Theorems \ref{thm:UpperBound} and \ref{t3} closely follows that of Nachmias and Peres\cite{np1} who applied some Martingale analysis,
including the Optional Stopping Theorem, to obtain a short elegant proof of what happens inside
the scaling window for $G_{n,p=c/n}$. See also \cite{np2} where they apply similar analysis to
also obtain a short proof of what happens outside the scaling window, including tight bounds on
the size of the largest component.

The approach we take for Theorem \ref{t2} is a first moment argument similar in spirit to one
applied in \cite{tl} to $G_{n,p}$, along with a very simple Martingale analysis.

\subsection{Our bound on $\D$}\label{sd}
Since there is a vertex $v$ of degree $\D$, we always have $R> \frac{d_v(d_v-2)^2}{2|E|}\geq \frac{\D^3}{18|E|}$.
Lemma \ref{lem:observations} in the next section gives $|E|\leq (1+\hf Q)n<2n$.  This yields $R>\frac{\D^3}{36n}$ and hence $\D<4n^{1/3}R^{1/3}$.
So our bound on $\D$ is within a factor of $O(\log n)$ of the maximum that $\D$ can possibly be when $|Q|=o(1)$.
In fact, it is possible to reduce this factor somewhat - our arguments still work if $\D\leq \k n^{1/3}R^{1/3}(\log n)^{-1/2}$ for some sufficiently small constant $\k$ that depends on $\la,\e,\z$.
But it can't be eliminated entirely:

Consider a degree sequence where one vertex $v$ has degree $\D\gg n^{1/3}$, all other vertices have small
degree and the contribution of those other vertices to $Q$ is $O(n^{-1/3})$; eg. $\frac{3}{4}$ of them have degree 1 and the others have degree 3.  Then we have $R=\Theta(\frac{\D^3}{n})$ and $Q=\Theta(\frac{\D^2}{n})=\Theta(n^{-1/3}R^{2/3})$ and so it is within what
our results say is the scaling window.  However, the same arguments that we use to prove
Theorem \ref{thm:UpperBound}(a) will prove that, with high probability, the random graph has a component
of size $\Theta(n^{2/3})\gg n^{2/3}R^{-1/3}$ (see the remark following that proof in Section \ref{stua}).

What causes that degree sequence to behave in this manner is that $R$ is large entirely because of a single vertex.
If we remove $v$, then the remaining degree sequence has $R=O(1)$.  Note that our bound on $\D$ in Condition D
is equivalent to $R>\frac{\D^3(\ln n)^3}{n}$ and so it is always satisfied if, eg., there are at least $(\ln n)^3$
vertices of degree $\D$.  Our bound on $\D$ can be viewed as a condition that the asymptotic order
of $R$ is determined by several high degree vertices. On the other hand, there are counterexamples when
it is determined by a small number of vertices.

\subsection{Related Work}\label{srw}
In 2000, Aiello, Chung and Lu\cite{acl} applied the results of Molloy and Reed\cite{mr1,mr2} to a model for
massive networks. They also extended those results to apply to power law degree sequences with maximum degree
higher than that required by \cite{mr1,mr2}.
Since then, that work been used numerous times to analyze massive network models
arising in a wide variety of fields such as physics, sociology and biology (see eg. \cite{nbw}).

Cooper and Frieze\cite{cf} proved, amongst other things, an analogue of the main results of \cite{mr1,mr2} in the
setting of giant strongly connected components in random digraphs.

Fountoulakis and Reed\cite{fr} extended the work of \cite{mr1} to degree sequences that do not satisfy the
convergence conditions required by \cite{mr1}.  They require $\D\leq |E|^{1/2-\e}$ which in their setting implies
$\D\leq O(n^{1/2-\e})$.

Kang and Seierstad\cite{ks} applied generating functions to
study the case where $Q=o(1)$, but is outside of
the scaling window. They require a maximum degree of at most $n^{1/4-\e}$ and that
the degree sequences satisfy certain conditions
that are stronger than those in \cite{mr1}; one of these conditions implies that $R$
is bounded by a constant. Based on what is known for $G_{n,p}$, it was natural
to guess that for $|Q|\gg n^{-1/3}$ we would
have $|\cm|\neq\Theta(n^{2/3})$.  They proved that if
$Q\ll -n^{-1/3}$ then $|\cm|\ll n^{2/3}$, and if $Q\gg n^{-1/3}\log n$ then $|\cm|\gg n^{2/3}$.
So for the case where $R=O(1)$ is bounded, this almost confirmed that natural guess -
except that they did not cover the range where $n^{-1/3}\ll Q=O(n^{-1/3}\log n)$.

Jansen and Luczak\cite{jl} used simpler techniques to obtain a result along the lines of that in \cite{ks}.
They require a maximum
degree of $n^{1/4}$, and they also require $R=O(1)$; in fact, they require $\inv{n}\sum_{v} d_v^{4+\eta}$
to be bounded by a constant (for some arbitrarily small constant $\eta>0$), but they conjecture that
having $\inv{n}\sum_{v} d_v^{3}$ bounded (i.e. $R$ bounded) would suffice.
For $Q\gg n^{-1/3}$, they prove that $|\cm|=\Theta(nQ)\gg n^{2/3}$.
Thus (in the case that their conditions hold) they eliminated the gap left over from \cite{ks}.
Furthermore, this also shows that the asymptotic order of $|\cm|$ increases with $|Q|$
in that range, thus eliminating the possibility of a scaling window extending into that range.
They also used their techniques to obtain a simpler proof of the main results from \cite{mr1,mr2}.

So for the case $R=O(1)$, the bound on the scaling window provided by Theorems \ref{t2}(a) and \ref{t3} was previously known  (under somewhat stronger conditions).  But it was not known that $|Q|=O(n^{-1/3})$ was
indeed the scaling window; it was possibly smaller.   In fact, it was not even clear that there was any scaling
window  in terms of $Q$ at all.  No bounds on $|\cm|$ were known for when $|Q|=O(n^{-1/3})$. And nothing was known for the case when $R$ grows with $n$.

\section{Preliminaries}
\subsection{Some Observations}
We start with two easy observations.
%We assume the conditions of Theorem~\ref{thm:UpperBound} throughout.

\begin{lemma}
\label{lem:observations}
%\mbox{  }\\
If $\cald$ satisfies Condition D(b,c) then:
\begin{enumerate}
\item [(a)] $\frac{1}{2}n \le |E| \le (1+\hf Q)n$, and $\sum_{u \in V} d_u^2=(4 +2Q)|E|$.
%\item [(b)] For sufficiently large $n$, $n_1  \ge \frac{\zeta}{2}|E|$.
\item [(b)] $\frac{\zeta}{4} \le R \le 2\Delta$.
\end{enumerate}
\end{lemma}

\begin{proof}

\noindent{\bf Part (a)}
$$2|E| Q  =-4|E|+\sum_{u \in V} d_u^2,$$
which establishes the second assertion. Now by the Cauchy-Schwarz
inequality,
$$2|E|=\sum_{u \in V} d_u \le \sqrt{n} \sqrt{\sum_{u \in V} d_u^2} \le \sqrt{n} \sqrt{4|E|(1+\hf Q)},$$
which shows that $|E| \le (1+\hf Q)n$.  The fact that every vertex has degree at least one implies $|E|\geq\frac{1}{2}n$.

%\noindent{\bf Part (b)}
%$$\sum_{u: d_u>2} d_u -n_1 \leq \sum_{u \in V} d_u(d_u-2)=2|E|Q,$$
%
%and by part (a),
%
%$$\sum_{u: d_u>2} d_u + n_1 = 2|E| - 2 n_2   \ge  2|E|-(1-\z)n.$$
%
%Subtracting the former from the latter yields $2n_1 \ge 2|E|-(1-\z)n-2|E|Q$.
%Since $|E|\leq 2n$ by part (a), this gives $n_1\ge |E|(\z-Q)\geq|E|\frac{\z}{2}$.

\noindent{\bf Part (b)}
\[R = \sum_{u \in V} \frac{d_u (d_u-2)^2}{2|E|} \ge \sum_{u: d_u \neq
2} \frac{d_u}{2|E|} \ge \frac{\zeta n}{2|E|} \ge \frac{\zeta}{4}.\]
On the other hand for sufficiently large $n$,
\begin{eqnarray*}
R&=&\sum_{u}\frac{d_u(d_u-2)^2}{2|E|}
\le (\Delta-2) \left(\sum_{u:d_u>1}\frac{d_u(d_u-2)}{2|E|}\right)+\frac{n_1}{2|E|}\\
&\le& (\Delta-2) (Q+\frac{n_1}{2|E|})+\frac{n_1}{2|E|}\leq(\D-2)(Q+1)+1 < 2\Delta,
\end{eqnarray*}
since $|Q|<1$.
\end{proof}
\subsection{The Random Model \label{sec:Model}}

In order to generate a random graph with a given degree sequence
${\cal D}$, we use the {\em configuration model} due to Bollob\'as\cite{bb}
and inspired by Bender and Canfield\cite{bc}.  In particular, we:
\begin{itemize}
\item Form a set $L$ which contains $d_v$ distinct copies of every
vertex $v$.

\item Choose a random perfect matching over the elements of $L$.

\item Contract the different copies of each vertex $v$ in $L$ into a single vertex.
\end{itemize}
This may result in a graph ${\cal G}({\cal D})$ with
multiple edges and loops, but our conditions on $\cald$ imply that for $n$ sufficiently large,
the probability that ${\cal G}({\cal D})$ is simple is bounded away from zero.
%then the expected number of loops and multiple edges
%is also bounded by a constant, and as shown in~\cite{????}, the probability
%that ${\cal G}({\cal D})$ is simple is at least $\r=\r(.
Furthermore if one
conditions on ${\cal G}({\cal D})$  being simple, then it is
uniformly distributed over the simple graphs with degree sequence
${\cal D}$. This allows us to translate results about ${\cal G}({\cal D})$
to results about a uniform simple graph with degree sequence $\cald$.

\begin{proposition}\label{pcm}
Consider any degree sequence $\cald$ satisfying Condition D(b,c,d).
Suppose that a property $\calp$ holds with
probability at most $\e$ for a uniformly random configuration with degree sequence $\cald$.  Then
for a uniformly random graph with
degree sequence $\cald$, $\pr(\calp)\leq \e \times e$.
\end{proposition}

\begin{proof}
\noindent Let $\calh$ be a random configuration with degree sequence $\cald$.
Lemma \ref{lem:observations}(b) implies that $\D=o(|E|^{1/2})$. This allows us to apply
Corollary 1.5 of \cite{sj}, which states that the probability of our configuration
being simple is
\[e^{\frac{1}{4}-\frac{1}{4}\left(\frac{\sum_{v\in G}d_v^2}{2|E|}\right)^2}+o(1).\]
Lemma \ref{lem:observations}(a) implies that this is at most
\[e^{\frac{1}{4}-\frac{1}{4}\left(\frac{4+2Q}{2+Q}\right)^2}+o(1)=e^{-3/4}+o(1)>e^{-1},\]
for $n$ sufficiently large.
The probability that a random graph
with degree sequence $\cald$ has $\calp$ is:
\[\pr(\calh \mbox{ has }\calp|\calh \mbox{ is simple})
\leq\pr(\calh \mbox{ has }\calp)/\pr(\calh \mbox{ is simple})\leq \e \times e.\]
\end{proof}

\subsection{Martingales}
A random sequence $X_0,X_1,...$ is a {\em martingale} if for all
$i\geq0$, $\ex(X_{i+1}|X_0,...,X_i)=X_i$.  It is a {\em submartingale}, resp. {\em supermartingale},
if for all $i\geq0$, $\ex(X_{i+1}|X_0,...,X_i)\geq X_i$, resp. $\ex(X_{i+1}|X_0,...,X_i)\leq X_i$.

A {\em stopping time} for a random sequence $X_0,X_1,...$ is a step $\t$ (possibly $\t=\infty$)
such that we can determine whether $i=\t$ by examining only $X_0,...,X_i$.
%We say that
%a stopping time $\t$ is {\em bounded} if $\t$ is always finite; i.e. $\pr(\t=\infty)=0$.
It is often useful to view a sequence as, in some sense, halting at time $\t$; a convenient
way to do so is to consider the sequence $X_{\min(i,\t)}$, whose $i$th term is $X_i$ if
$i\leq\t$ and $X_{\t}$ otherwise.

In our paper, we will make heavy use of the Optional Stopping Theorem.
The version that we will use is the following, which is implied by Theorem 17.6 of \cite{ypbook}:

\noindent {\bf The Optional Stopping Theorem} {\em Let $X_0,X_1,...$ be a  martingale
(resp. submartingale, supermartingale), and let $\t\geq 0$
be a stopping time. If there is a fixed bound $T$ such that
$\pr(\t\leq T)=1$ then
$\ex(X_{\t})=X_0$ (resp. $\ex(X_{\t})\geq X_0$, $\ex(X_{\t})\leq X_0$).}

We will also use the following concentration theorem, which is given by
Theorems 6.1 and 6.5 from \cite{cl}.
\begin{theorem}
\label{thm:MartingaleConc} Let $X_0,X_1,...$ be a martingale  satisfying
\begin{enumerate}
\item[(a)] $\Var(X_i|X_0,X_1,...,X_{i-1}) \le \sigma_i^2$, for $1 \le i \le n$.
\item[(b)] $|X_i-X_{i-1}| \le M$, for $1 \le i \le n$.
\end{enumerate}
Then
$$\Pr[|X-\Ex X| \ge \r] \le 2e^{-\frac{\r^2}{2(M\r+\sum \sigma_i^2)}}.$$
\end{theorem}

\subsection{The Branching Process}\label{s24}
As in \cite{mr1}, we will examine our random graph using a branching process
of the type first applied to random graphs by Karp in \cite{rk}.
This time, we need to be much more careful,
since the branching parameter is $Q+1$ which can be $1+o(1)$.

Given a vertex $v$, we explore the graph ${\cal G}({\cal D})$
starting from $v$ in the following manner. At step $t$, we will have a partial subgraph $C_t$
which has been exposed so far.  Typically, there will be some vertices of $C_t$ whose
neighbours have not all been exposed; since we are working in the configuration model,
this is equivalent to saying that there are some vertex-copies of vertices in $C_t$
whose partners in the configuration have not been exposed.  We choose one of those vertex-copies
and expose its partner by selecting that partner uniformly at random from amongst all
vertex-copies that are still unmatched; if the partner is a vertex-copy of a vertex $u\notin C_t$,
then we add $u$ to $C_t$. This yields an edge of $C_t$. If all the vertex-copies of all vertices in $C_t$
are matched, then this indicates that we have exposed an entire component.  So we start
exploring a new component beginning with an arbitrary vertex.  Note that $C_t$ may contain several
components, but that all vertices with unmatched vertex-copies belong to the same component - the one
that is currently being explored.

We will use $Y_t$ to denote the total number of unmatched vertex-copies of vertices in $C_t$.
So $Y_t=0$ indicates that we have exposed an entire component and are about to start a new one.

\begin{enumerate}
\item Choose an arbitrary vertex $v$ and initialize $C_0=\{v\}$; $Y_{0}=\deg(v)$.
\item Repeat while there are any vertices not in $C_t$:
\begin{enumerate}
\item If $Y_t=0$, then pick a uniformly random vertex-copy from amongst all unmatched vertex-copies;
let $u$ denote the vertex of which it is a copy. $C_{t+1}:=C_t\cup\{u\}$; $Y_{t+1}:=\deg(u)$.
\item Else choose an arbitrary unmatched vertex-copy of any vertex $v\in C_t$.
Pick a uniformly random vertex-copy from amongst all other unmatched vertex-copies;
let $u$ denote the vertex of which it is a copy. Match these two vertex-copies; thus exposing
$uv$ as an edge of $C_{t+1}$.
\begin{enumerate}
\item If $u\notin C_t$ then $C_{t+1}:=C_t\cup\{u\}$; $Y_{t+1}:=Y_t+\deg(u)-2$.
\item Else $C_{t+1}:=C_t$; $Y_{t+1}:=Y_t-2$.
\end{enumerate}
\end{enumerate}
\end{enumerate}

Note that $C_t$ is a Markov process, and hence $Y_t$ depends only on
$C_{t-1}$, and not on the way that $C_{t-1}$ is exposed in the
branching process. For $t \ge 0$ let
\begin{itemize}
\item $\eta_{t+1} := Y_{t+1}-Y_{t}$.

\item $D_t:=Y_t+\sum_{u \not\in C_t}d_u$, the total number of unmatched vertex-copies remaining at time $t$.

\item $v_t:=\emptyset$ if $C_{t-1}$ and $C_t$ have the same vertex set, and if not, then $v_t$ is the unique
vertex in $C_t \setminus C_{t-1}$.

\item $Q_t:=\frac{\sum_{u \not\in C_t}d_u^2}{D_t - 1}-2$, and $R_t := \frac{4(Y_t-1)+\sum_{u \not\in C_t}d_u(d_u-2)^2}{D_t -1}$.
\end{itemize}

{\bf Remark:} Since every vertex has degree at least one, $D_t>0$ for every $t$ until the procedure halts.
Since the total number of vertex-copies is even, $D_t$ is even.  Therefore $D_t>1$ and so
$Q_t,R_t$ are well-defined.

Note that $Q_t$ and $R_t$ begin at $Q_0\approx Q$ and $R_0\approx R$. Furthermore,
for $u \not\in C_t$, $\Pr[v_{t+1}=u]=\frac{d_u}{D_t-1}$, and so if $Y_t>0$ then
the expected change in $Y_t$ is
\begin{equation}
\label{eq:Qt} \Ex[\eta_{t+1} | C_t] = (\sum_{u \not\in C_t}
\Pr[v_{t+1}=u]\times d_u)-2= \frac{\sum_{u \not\in C_t}d_u^2}{D_t - 1}-2=Q_t.
\end{equation}
If $Q_t$ remains approximately $Q$, then $Y_t$ is a random walk with drift approximately $Q$.
So if $Q<0$ then we expect $Y_t$ to keep returning to zero quickly, and hence we only discover
small components.  But if $Q>0$ then we expect $Y_t$ to grow large; i.e. we expect to discover a
large component.  This is the intuition behind the main result of \cite{mr1}.

The parameter $R_t$ measures the expected value of the square of the change in $Y_t$, if $Y_t>0$:
\begin{equation}
\label{eq:Rt} \Ex[\eta_{t+1}^2 | C_t] = \Pr[v_{t+1}=\emptyset]\times 4
+ \sum_{u \not\in C_t} \Pr[v_{t+1}=u] \times(d_u-2)^2=
\frac{4(Y_t-1)+\sum_{u \not\in C_t}d_u(d_u-2)^2}{D_t - 1}=R_t.
\end{equation}

If $Y_t=0$, then the
expected values of $\eta_{t+1}$ and $\eta_{t+1}^2$ are not equal to $Q_t,R_t$, as in this case
we have
\begin{equation}
\label{eq:etatFromZero} \Ex[\eta_{t+1} | C_t] =
\frac{\sum_{u \not\in C_t}d_u^2}{D_t},
\end{equation}
and, recalling from the above remark that $D_t>1$,
\begin{equation}
\label{eq:etat2FromZero}
\Ex[\eta_{t+1}^2 | C_t]= \frac{\sum_{u\not\in C_t}d_u^3}{D_t}\geq R_t\times\frac{D_t-1}{D_t}\geq\frac{R_t}{2}.
\end{equation}

Note that, for $Y_t>0$, the expected change in $Q_t$ is  approximately:
\[\Ex[Q_{t+1}-Q_t| C_t] \approx -\sum_{u \not\in C_t} \Pr[v_{t+1}=u] \times\frac{d_u^2}{D_t-1}
=-\frac{\sum_{u \not\in C_t}d_u^3}{(D_t - 1)^2}\]
which, as long as $D_t=n-o(n)$, is asymptotically of the same order as $-\frac{R_t}{n}$.
So if $R_t$ remains approximately $R$, then $Q_t$ will have a drift of roughly $-\frac{R}{n}$;
i.e. the branching factor will decrease at approximately that rate.
So amongst degree sequences with the same value of $Q$, we should expect those with large $R$
to have $|\cm|$ smaller. This explains why $|\cm|$ is a function of both $Q$ and $R$
in Theorem \ref{thm:UpperBound}.

Finally, note that since  $D_t$ decreases by at
most $2$ during any one step, we have
\begin{equation}
\label{eq:HalfEdges}
D_t \ge 2|E| - 2t.
\end{equation}

%
%\begin{theorem}
%\label{thm:UpperBound} Suppose that $|Q| \le \lambda n^{-1/3} \lceil
%R \rceil^{-1/3}$, $\Delta \le n^{2/7} (\ln n)^{-1} $, and
%$\lambda_2<c<1$, for constants $c,\lambda>0$. Then for every
%constant $\epsilon>0$, there exists a constant $C>1$ such that
%$$\Pr[|{\cal C}_{\rm max}| \ge C n^{2/3} R^{-1/3}  ] \le \epsilon,$$
%where ${\cal C}_{\rm max}$ is the largest component of ${\cal
%G}(D)$.
%\end{theorem}

%\begin{theorem}
%\label{thm:LowerBound} Suppose that $|Q| \le \lambda n^{-1/3} \lceil
%R \rceil^{2/3}$, $\Delta \le n^{2/7} (\ln n)^{-1}$, and
%$\lambda_2<c<1$, for constants $c,\lambda>0$. Then for every
%constant $\epsilon>0$, there exists a constant $\delta>0$ such that
%$$\Pr[|{\cal C}_{\rm max}| \le \delta n^{2/3} R^{-1/3} ] \le \epsilon,
%$$
%where ${\cal C}_{\rm max}$ is the largest component of ${\cal
%G}(D)$.
%\end{theorem}

\section{Concentration of $Q_t$ and $R_t$}\label{scon}
In this section, we estimate the expected values of
$Q_t$ and $R_t$ and show that they are concentrated.
We begin with $R_t$.

\begin{lemma}
\label{lem:concR} For each
$1 \le t \le \frac{\z}{400}\frac{n}{\D}$,
$$\Pr[|R_t-R| \ge R/2] < n^{-10}.$$
\end{lemma}
\begin{proof} It would be convenient if $R_t,R_{t-1}$ had the same denominator. So
for $t \ge 1$ we define
$$\tilde{R}_{t}:= \frac{4(Y_{t}-1) + \sum_{u \not\in C_{t}} d_u (d_u-2)^2}{D_{t-1}-1}=R_t \frac{D_{t}-1}{D_{t-1}-1}.$$
Note that $|D_t-D_{t-1}| \le 2$ and $Y_t \le \Delta t$, and that for $t \le \frac{\z}{400}\frac{n}{\D}<\hf|E|$,
we have $D_t,D_{t-1}>|E|$ by (\ref{eq:HalfEdges}).
Hence, applying Lemma \ref{lem:observations}, we obtain that for sufficiently large $n$,
\begin{eqnarray}
\nonumber |R_t-\tilde{R}_t|&=&(4(Y_t-1) +
\sum_{u \not\in C_t} d_u (d_u-2)^2)\left|\inv{D_{t}-1}-\inv{D_{t-1}-1}\right|\\
\nonumber &=&(4(Y_t-1) + \sum_{u \not\in C_t} d_u (d_u-2)^2)\frac{\left|D_{t}-D_{t-1}\right|}
{(D_{t}-1)(D_{t-1}-1)}\\
\label{eq:RTildeDiff}
&\le& 2\frac{4(Y_t-1) +
\sum_{u \not\in C_t} d_u (d_u-2)^2}{|E|^2}
\le \frac{8\Delta\times \frac{\z}{400}\frac{n}{\D}}{|E|^2}+\frac{4R}{|E|}
< \frac{5R}{|E|}.
\end{eqnarray}
Using again the fact that $D_{t-1}-1\geq |E|$, we have for $n$
sufficiently large:
\begin{eqnarray*}
\left| \Ex\left[\tilde{R}_t-R_{t-1}|C_{t-1}\right] \right|&=&
\left|\Ex\left[\frac{4(Y_t-Y_{t-1})}{D_{t-1}-1}|C_{t-1}\right] -
\sum_{u \not\in C_{t-1}} \Pr[v_t=u] \frac{d_u(d_u-2)^2}{D_{t-1}-1}\right| \\
&\le& \frac{4\Delta}{|E|} +  \sum_{u \not\in C_{t-1}}
\frac{d_u^2(d_u-2)^2}{(D_{t-1}-1)^2} \le \frac{4\Delta}{|E|} +
\frac{2\Delta R}{|E|} < \frac{20}{\z}\frac{\Delta R}{|E|},
\end{eqnarray*}
which together with (\ref{eq:RTildeDiff}) shows
\begin{equation}
\label{eq:expectedChangeRt} \left|
\Ex\left[R_t-R_{t-1}|C_{t-1}\right] \right|
\le \frac{20}{\z}\frac{\Delta R}{|E|} + \frac{5R}{|E|}
\le \frac{40}{\z}\frac{\Delta R}{|E|}.
\end{equation}
Using the fact that $(a+b)^2\leq 2a^2+2b^2$ we similarly obtain:
\begin{eqnarray*}
\Ex\left[\left| \tilde{R}_t-R_{t-1}\right|^2 |C_{t-1}\right]
&\le& 2\Ex\left[\left(\frac{4(Y_t-Y_{t-1})}{D_{t-1}-1}\right)^2|C_{t-1}\right]
+ 2\sum_{u \not\in C_{t-1}} \Pr[v_t=u]
\left(\frac{d_u(d_u-2)^2}{D_{t-1}-1}\right)^2
\\
&\le& 2\frac{16\Delta^2}{|E|^2} +  2 \sum_{u \not\in C_{t-1}} \frac{d_u^3(d_u-2)^4}{(D_{t-1}-1)^3}
< \frac{32\Delta^2}{|E|^2} + \frac{4\Delta^4 R}{|E|^2}
<\frac{150}{\z}\frac{\Delta^4 R}{|E|^2}.
\end{eqnarray*}
We conclude from this, (\ref{eq:RTildeDiff}) and
Lemma~\ref{lem:observations}(b) that
\begin{eqnarray}
\nonumber \Var[R_t|C_{t-1}] &\le&
\Ex\left[|R_t-R_{t-1}|^2|C_{t-1}\right] \le
2\Ex\left[|\tilde{R}_t-R_{t-1}|^2|C_{t-1}\right] + 2
\Ex\left[|R_t-\tilde{R}_t|^2|C_{t-1}\right] \\
\label{eq:RboundDiffSquare}  &\le& 2\frac{150}{\z}\frac{\Delta^4 R}{|E|^2}
+ 2(5)^2\frac{R^2}{|E|^2}
\le \frac{300}{\z}\frac{\Delta^4 R}{|E|^2}
+100\frac{R \Delta}{|E|^2} < \frac{500}{\z}\frac{\Delta^4 R}{|E|^2}.
\end{eqnarray}
Note that by (\ref{eq:expectedChangeRt}) and the bound $\D\leq n^{1/3}R^{1/3}/\ln n$, for $0 \le t \le \frac{\z}{400}\frac{n}{\D}$ and $n$ sufficiently large, we have:
\begin{equation}
\label{eq:expectedRt} |R-\Ex R_t| \le |R-R_0| + |R_0-\Ex R_t| \le
\frac{4\Delta+\Delta(\Delta-2)^2}{2|E|}+t\times\frac{40}{\z}\frac{\Delta R}{|E|}
< R/4.
\end{equation}
Applying (\ref{eq:RTildeDiff}), we have
\begin{equation}
\label{eq:boundMaxChangeR} |R_t-R_{t-1}| \le |R_t-\tilde{R}_t| +
|\tilde{R}_t-R_{t-1}| \le \frac{5 R}{|E|}+
\frac{4\Delta + \Delta(\Delta-2)^2}{D_{t-1}-1} < \frac{10 \Delta^3}{|E|}.
\end{equation}
Now by (\ref{eq:RboundDiffSquare}), (\ref{eq:expectedRt}), (\ref{eq:boundMaxChangeR}), and
Theorem~\ref{thm:MartingaleConc} and the bound $\D\leq n^{1/3}R^{1/3}/\ln n$:
\begin{eqnarray*}
\Pr\left[|R_t -R| \ge \frac{R}{2}\right] &\le& \Pr\left[|R_t -\Ex
R_t| \ge \frac{R}{4}\right] \le
e^{-\frac{(R/4)^2}{2\left(\frac{R}{4} \frac{10\Delta^3}{|E|} + t\frac{500}{\z}
\frac{\Delta^4 R}{|E|^2}\right)}} <
e^{-\frac{R|E|}{200\Delta^3 }} \\
&<& e^{-(\ln n)^3}<n^{-10},
\end{eqnarray*}
for sufficiently large $n$.
\end{proof}

%The next lemma shows that, under our hypothesis, with very high
%probability in the first steps $Q_t$ will remain in a certain range.

Next we turn to $Q_t$:

\begin{lemma}
\label{lem:concQ} For each  $1 \le t \le \frac{\z}{1000}\frac{|Q|n}{R}+2n^{2/3}R^{-1/3}$,
\[\Pr\left[|Q_t-Q|>\hf |Q|+ \frac{800}{\z}n^{-1/3} R^{2/3}\right] \le n^{-10}.\]
%\begin{equation}
%\label{eq:concentration} \Pr\left[|Q_t-Q|>\hf Q\right] \le
%e^{-n^{1/30}}
%\end{equation}
%where $${\cal I}:=\left[-400  n^{-1/3}R^{2/3},400 n^{-1/3} R^{-1/3}\right].$$
\end{lemma}

%{\bf Remark:} As we will see, $Q_t$ has a drift of order $-\frac{R}{n}$. So one of the restrictions
%on $t$ is that it be small enough for $\Ex Q_t$ to be at least $\frac{3}{4}Q + 400n^{-1/3} R^{2/3}$.

\begin{proof}
Again, we make the denominators the same by setting
$$\tilde{Q}_{t}:= \frac{\sum_{u \not\in C_t} d_u^2}{D_{t-1}-1}-2.$$
%
%Note that for $t \le C n^{2/3} R^{-1/3}$,
%$D_t,D_{t-1}>|E|$,  $|D_t-D_{t-1}| \le 2$.
Using the same argument as for (\ref{eq:RTildeDiff}), and applying Lemma \ref{lem:observations}(a), we obtain that for sufficiently
large $n$,
\begin{equation}
\label{eq:diffQQtilde} |Q_t-\tilde{Q}_t|
\le 2\frac{\sum_{u \not\in C_t} d_u^2}{|E|^2} \le \frac{4|Q|+8}{|E|} < \frac{10}{|E|}.
\end{equation}
Trivially
\begin{equation}
\label{eq:boundMaxChange} |Q_t-Q_{t-1}| \le
|Q_t-\tilde{Q}_t|+|\tilde{Q}_t-Q_{t-1}|\le
\frac{\Delta^2}{|E|}+\frac{10}{|E|} < \frac{20 \Delta^2}{|E|}.
\end{equation}
In what follows, we use the facts that $2|E|\geq D_{t-1}-1\geq 2|E|-2t+1>|E|$ and $D_{t-1}\geq Y_{t-1}$.
Note that whether  $Y_{t-1}>0$ or $Y_{t-1}=0$, for sufficiently
large $n$, by Lemma~\ref{lem:observations}(a) we always have
\begin{eqnarray*}
\Ex\left[\tilde{Q}_t-Q_{t-1}|C_{t-1}\right]
&=& - \sum_{u \not\in C_{t-1}} \Pr[v_t=u] \frac{d_u^2}{D_{t-1}-1}
= - \sum_{u \not\in C_{t-1}} \frac{d_u^3}{(D_{t-1}-1)^2}  \\
&<& -  \frac{4(Y_{t-1}-1) +\sum_{u \not\in C_{t-1}} d_u(d_u-2)^2-4(Y_{t-1}-1)}{(D_{t-1}-1)^2}
%+ \frac{4(Y_{t-1}-1)
%+\sum_{u \not\in C_{t-1}}
%4d_u-4d_u^2}{(D_{t-1}-1)^2}
\\
&\le& -\frac{R_{t-1}-4}{D_{t-1}-1} \le -\frac{R_{t-1}}{2|E|} +\frac{4}{|E|}.
\end{eqnarray*}
Combining this with (\ref{eq:diffQQtilde}) we have
\begin{equation}
\label{eq:diffQ} \Ex\left[Q_t-Q_{t-1}|C_{t-1}\right] \le
-\frac{R_{t-1}}{2|E|}  + \frac{4}{|E|}+\frac{10}{|E|}
< -\frac{R_{t-1}-30}{2|E|}.
\end{equation}
For $n$, and hence $|E|$, sufficiently large:
\[Q_0-Q=\frac{d_v^2}{2|E|-1}+\sum_{u\in G}d_u^2\left(\frac{1}{2|E|-1}-\frac{1}{2|E|}\right)
=\frac{d_v^2}{2|E|-1}+\frac{(Q+2)}{2|E|-1}<\frac{\D^2}{|E|}.\]
Now (\ref{eq:expectedRt}), (\ref{eq:diffQ}),
Lemma~\ref{lem:observations}(a,b) and the bound $\D\leq n^{1/3}R^{1/3}/\ln n$ imply that for
$t \le \frac{\z}{1000}\frac{|Q|n}{R}+2n^{2/3}R^{-1/3}$,
\begin{eqnarray}
\nonumber \Ex[Q_t-Q] &\le& |Q-Q_0|+ \Ex[Q_t-Q_0] \le
\frac{\Delta^2}{|E|}+\left(\frac{30 t}{2|E|}
- \sum_{j=0}^{t-1} \Ex \frac{R_j}{2|E|}\right)\leq \frac{\Delta^2}{|E|}+\frac{30 t}{2|E|}\\
\label{eq:ExQtAbove} &\le&
2n^{-1/3}R^{2/3}(\ln n)^{-2}+\frac{|Q|}{4}+60n^{-1/3}R^{-1/3}
\leq \frac{|Q|}{4}+\frac{400}{\z}n^{-1/3}R^{2/3}.
%2\Delta^2 n^{-1}+100 n^{-1/3} R^{-1/3}
%- \frac{R t}{3n} < 200 n^{-1/3} R^{-1/3} - \frac{R t}{3n}.
\end{eqnarray}
Furthermore
\begin{eqnarray*}
\Ex\left[\tilde{Q}_t-Q_{t-1}|C_{t-1}\right]
&=& - \sum_{u \not\in C_{t-1}} \frac{d_u^3}{(D_{t-1}-1)^2}  \\
&\ge& -\frac{ 9\sum_{u\not\in C_{t-1}} d_u(d_u-2)^2
+ \sum_{u:d_u=2} 8}{|E|^2} \\
&>&  -\frac{18 R + 8}{|E|} \geq -\frac{50}{\z}\frac{R}{|E|},
\end{eqnarray*}
which together with (\ref{eq:diffQQtilde}) shows that
$$\Ex\left[Q_t-Q_{t-1}|C_{t-1}\right] \ge -\frac{50}{\z}\frac{R}{|E|} - \frac{10}{|E|}
\geq -\frac{90}{\z}\frac{R}{|E|},$$
and hence for $1 \le t \le \frac{\z}{1000}\frac{|Q|n}{R}+2n^{2/3}R^{-1/3}$, using
the bound $\D\leq n^{1/3}R^{1/3}/\ln n$,
\begin{equation}
\label{eq:ExQtBelow} \Ex[Q_t-Q] \ge \Ex[Q_t-Q_0] - |Q-Q_0|
\ge - \frac{90}{\z}\frac{t R}{|E|} - \Delta^2 n^{-1}
> -\frac{|Q|}{4}-\frac{400}{\z}n^{-1/3}R^{2/3}.
\end{equation}
Using similar arguments we obtain
\begin{eqnarray*}
\Ex\left[|\tilde{Q}_t-Q_{t-1}|^2|C_{t-1}\right] &=& \sum_{u \not\in
C_{t-1}} \Pr[v_t=u] \frac{d_u^4}{(D_{t-1}-1)^2} = \sum_{u \not\in
C_{t-1}}  \frac{d_u^5}{(D_{t-1}-1)^3}  \\ &\le& \Delta^2 \sum_{u
\not\in C_{t-1}}  \frac{d_u^3}{|E|^3} \le
\frac{\Delta^2}{|E|^2}\frac{ 9\sum_{u \not\in C_{t-1}} d_u(d_u-2)^2
+ \sum_{u:d_u=2} 8}{|E|} \\
&\leq& \frac{50}{\z}\frac{\Delta^2 R}{|E|^2}  \le \frac{200}{\z}\frac{\Delta^2 R}{n^2}.
\end{eqnarray*}

As in (\ref{eq:RboundDiffSquare}), this, (\ref{eq:diffQQtilde}) and
Lemma~\ref{lem:observations}(b) yield
\begin{eqnarray}
\nonumber \Var[Q_t|C_{t-1}] &\le&
\Ex\left[|Q_t-Q_{t-1}|^2|C_{t-1}\right] \le
2\Ex\left[|\tilde{Q}_t-Q_{t-1}|^2|C_{t-1}\right] + 2
\Ex\left[|Q_t-\tilde{Q}_t|^2|C_{t-1}\right] \\
\label{eq:QboundDiffSquare}
&\le& 2 \frac{200}{\z}\frac{\Delta^2 R}{n^2}
+ 2\frac{100}{|E|^2}<\frac{2000}{\z}\frac{\Delta^2 R}{n^2}.
\end{eqnarray}

By (\ref{eq:ExQtAbove}), (\ref{eq:ExQtBelow}) and (\ref{eq:boundMaxChange}), we can apply
Theorem~\ref{thm:MartingaleConc} with
$\r= \frac{1}{4}|Q|+\frac{400}{\z}n^{-1/3} R^{2/3}$
and $M=\frac{20 \Delta^2}{|E|}<40 n^{-1/3} R^{2/3}(\ln n)^{-2}$.
Similarly, (\ref{eq:QboundDiffSquare}) allows us to take $\s_i^2=\frac{2000}{\z}\frac{\Delta^2 R}{n^2}
<\frac{2000}{\z} n^{-4/3} R^{5/3}(\ln n)^{-2}$.
This yields:
\begin{eqnarray*}
%\label{eq:boundQtFromAbove}
\Pr\left[|Q_t-Q| \ge\hf |Q|+ \frac{800}{\z}n^{-1/3} R^{2/3}\right]
&\le& 2 e^{-\frac{(\frac{1}{4}|Q|+\frac{400}{\z}n^{-1/3} R^{2/3})^2}
{2\left(40 n^{-1/3} R^{2/3} \times
(\frac{1}{4}|Q|+\frac{400}{\z}n^{-1/3} R^{2/3})
+ t \times \frac{2000}{\z} n^{-4/3} R^{5/3}\right)(\ln n)^{-2}}}\\
&\leq& 2 e^{-\frac{(\frac{1}{4}|Q|+\frac{400}{\z}n^{-1/3} R^{2/3})^2(\ln n)^2}
{80 n^{-1/3} R^{2/3} \times
(\frac{1}{4}|Q|+\frac{400}{\z}n^{-1/3} R^{2/3})
+ 4|Q|n^{-1/3}R^{2/3}+\frac{8000}{\z} n^{-2/3} R^{4/3}}}.
\end{eqnarray*}
To bound this, note that:
\[\frac{(\frac{1}{4}|Q|+\frac{400}{\z}n^{-1/3} R^{2/3})^2}
{80 n^{-1/3} R^{2/3} \times
(\frac{1}{4}|Q|+\frac{400}{\z}n^{-1/3} R^{2/3})}=\frac{\frac{1}{4}|Q|+\frac{400}{\z}n^{-1/3} R^{2/3}}
{80 n^{-1/3} R^{2/3}}\geq\frac{5}{\z}.\]
\[\frac{(\frac{1}{4}|Q|+\frac{400}{\z}n^{-1/3} R^{2/3})^2}{4|Q|n^{-1/3}R^{2/3}}
\geq\frac{\frac{400}{\z}n^{-1/3} R^{2/3}}{16n^{-1/3}R^{2/3}}=\frac{25}{\z}.\]
\[\frac{(\frac{1}{4}|Q|+\frac{400}{\z}n^{-1/3} R^{2/3})^2}{\frac{8000}{\z} n^{-2/3} R^{1/3}}
\geq\frac{(n^{-1/3} R^{2/3})^2}{\z n^{-2/3} R^{4/3}}=\frac{1}{\z}.\]
These yield:
\[\Pr\left[|Q_t-Q| \ge\hf |Q|+ \frac{800}{\z}n^{-1/3} R^{2/3}\right]
\le 2 e^{-\frac{(\ln n)^2}{3\z}}<n^{-10},\]
for $n$ sufficiently large.
\end{proof}

\section{Proof of Theorem \ref{t2}}
We start by analyzing the subcritical phase; i.e. when $Q<-\omega(n) n^{-1/3}R^{2/3}$, where $\omega(n)$ grows with $n$.

First we show that with high probability, there are no components of size greater than $O(\sqrt{n/|Q|})$.
The proof will be a simple application of the Optional Stopping Theorem.

{\bf Proof of Theorem \ref{t2}(a).}  Fix any $\e>0$ and set $B=\frac{4}{\sqrt{\e}}$ and
$T=B\sqrt{n/|Q|}$.
For a given vertex $v$, we will bound the probability
that $v$ lies in a component of size at least $T$ by analyzing the branching process beginning at vertex $v$
and bounding the  probability that $Y_t$ does not return to zero before time $T$.

Note that for $n$ sufficiently large, $T<n^{2/3}R^{-1/3}$.
So Lemma \ref{lem:concQ}, implies that, with high probability,
\[Q_t\leq\hf Q+\frac{800}{\z}n^{1/3}R^{2/3}<\inv{4}Q,\]
for every $t\leq T$.
We define the stopping time
\[\g:= \min \{t: \mbox{$(Y_t=0)$, $(Q_{t}>\inv{4}Q)$ or $(t=T)$}\}.\]
Lemma \ref{lem:concQ} will show us that, with high probability, we will not
have $Q_{\g}>\inv{4}Q$.  So by upper bounding $\pr(\g=T)$, we can
obtain a good lower bound on $\pr(Y_t=0)$ which, in turn, is a lower bound on
the probability that $Y_t$ reaches zero before
time $T$.

For $t \le \gamma$, we have $Q_{t-1}\leq\inv{4}Q$. We also have $Y_{t-1}>0$ and so $\ex(\eta_t)$ is as in (\ref{eq:Qt}). Therefore we have:
\[\ex(Y_t-Y_{t-1})=Q_{t-1}\leq\inv{4} Q,\]
and so $Y_{\min(t,\g)}-\inv{4} Q\min(t,\g)$ is a supermartingale.  Applying the Optional Stopping Theorem
to $Y_{\min(t,\g)}-\inv{4} Q\min(t,\g)$ with stopping times $\t:=\g$
and bound $T:=T$ yields that
\[\ex(Y_{\g}-\inv{4} Q\g)\leq Y_0=d_v.\]
Since $Q<0$, this implies:
\[\ex(\g)\leq\frac{4(d_v-\ex(Y_{\g}))}{|Q|}\leq\frac{4d_v}{|Q|},\]
and so $\pr(\g=T)\leq\frac{4d_v}{|Q|T}$.  By Lemma \ref{lem:concQ}, $\pr(Q_{\g}>\inv{4}Q)<n^{-10}$
and so:
\[\pr(Y_{\g}\neq 0)\leq\frac{4d_v}{|Q|T}+n^{-10}<\frac{5d_v}{|Q|T},\]
for $n$ sufficiently large.

Let $Z$ be the number of vertices lying in components of size at least $T$.
Recalling that $\sum_v d_v=2|E|<3n$ by Lemma~\ref{lem:observations}(a), we have
\begin{eqnarray*}
\Pr[|\cm| \ge T] &\le& \Pr[Z \ge T]  \le  \frac{\Ex[Z]}{T}
\le \frac{1}{T} \sum_{v \in V} \Pr[{\cal C}_v \ge T] \\
&\le& \frac{1}{T} \sum_{v \in V}\frac{5d_v}{|Q|T}<\frac{16n}{|Q|T^2}=\frac{16}{B^2}=\e.
\end{eqnarray*}

This proves that Theorem~\ref{t2}(a) holds for
a random configuration.  Proposition \ref{pcm} implies that it holds for a random graph.
\proofend

Next we show that the random graph will, with probability
at least $\e$, have no components with at least two cycles.
The following helpful fact bounds the probability that specific pairs of vertex-copies
are joined in our random configuration:

\begin{proposition}\label{pc}
Specify any $\ell$ pairs of vertex-copies.  The probability that those pairs are joined
is at most  $\frac{(|E|-1-\ell)!}{2^{\ell}(|E|-1)!}$.
\end{proposition}

\begin{proof}
The number of ways of pairing $2r$ points is
$\frac{(2r)!}{2^rr!}$.
So the ratio of the number of configurations
with those $\ell$ pairs joined to the total number of configurations is
\begin{eqnarray*}
\frac{(2|E|-2\ell)!/2^{|E|-\ell}(|E|-\ell)!}{(2|E|)!/2^{|E|}|E|!}
&=&\frac{2^{\ell}|E|(|E|-1)...(|E|-\ell+1)}{(2|E|)...(2|E|-2\ell+1)}\\
&=&\frac{1}{(2|E|-1)(2|E|-3)...(2|E|-2\ell+1)}\\
&<&\frac{1}{2^{\ell}(|E|-1)(|E|-2)...(|E|-\ell)}.
\end{eqnarray*}
\end{proof}

We will also use:
\begin{proposition}\label{psum} For any $w_1,...,w_n\geq 0$, the average over all
subsets $\{x_1,...,x_{\ell}\}\subset\{1,...,n\}$ of $\prod_{i=1}^{\ell}w_{x_i}$ is at
most the average over all $\ell$-tuples $(x_1,...,x_{\ell})\in\{1,...,n\}^{\ell}$
of $\prod_{i=1}^{\ell}w_{x_i}$.
\end{proposition}

\begin{proof}
It is trivially true if $n=1$.  Note that the average over
all $\ell$-tuples $(x_1,...,x_{\ell})\in\{1,...,n\}^{\ell}$
of $\prod_{i=1}^{\ell}w_{x_i}$ is equal to $n^{-\ell}\left(\sum_{i=1}^n w_i\right)^{\ell}$
and hence is determined by $\sum_{i=1}^n w_i$.  The proposition now follows from
the easy observations:  (i) the two averages are equal if $w_1=...=w_n$ and (ii)
if $w_i\leq w_j$ then replacing $w_i,w_j$ by $w_i-\e,w_j+\e$ decreases $w_iw_j$
and hence decreases the average over all
subsets $\{x_1,...,x_{\ell}\}\subset\{1,...,n\}$ of $\prod_{i=1}^{\ell}w_{x_i}$.
\end{proof}

{\bf Proof of Theorem \ref{t2}(b)} As noted by Karonski for the proof of the very similar
Lemma 1(iii) of \cite{tl}:
if a component contains at least two cycles then it must contain at least one of
the following two subgraphs:
\begin{itemize}
\item $W_1$ - two vertices $u,v$ that are joined by three paths, where the paths are
vertex-disjoint except for at their endpoints.
\item $W_2$ - two edge-disjoint cycles, one containing $u$ and the other containing $v$,
and a $(u,v)$-path that is edge-disjoint from the cycles.  We allow $u=v$ in which case the
path has length zero.
\end{itemize}
In particular, if it contains two cycles that share more than one vertex, then
it is easy to see that it must contain a pair of cycles that form $W_1$. And
if it contains two cycles that share at most one vertex,
then those cycles plus a shortest path between them must form $W_2$.

We will prove that the expected number of such subgraphs of size at most $\inv{4}n$
is less than $\hf \e$.  By part (a) (after rescaling $\e$), the probability that there
is any component of size greater than $\inv{4}n$ is less than $\hf\e$. This proves the
theorem. We begin with the expected number of $W_1$'s.

Specify $u,v$, the number of internal vertices on each path - $\ell_1,\ell_2,\ell_3$, and the
internal vertices - $x_1,...,x_{\ell_1+\ell_2+\ell_3}$, in order along the paths.
%For each $\ell_1,\ell_2,\ell_3$,
%the number of choices is less than $n^{\ell_1+\ell_2+\ell_3+2}$.
Next specify exactly which vertex-copies are paired to form the edges of $W_1$; the number
of choices is $d_u(d_u-1)(d_u-2)\times d_v(d_v-1)(d_v-2)\times\prod_{i=1}^{\ell_1+\ell_2+\ell_3}d_{x_i}(d_{x_i}-1)$.
Therefore, applying Proposition \ref{pc}, the expected number of such subgraphs is at most:
\begin{equation}\label{ew1}
\sum_{u,v}\sum_{\ell_1,\ell_2,\ell_3\geq0}\sum_{x_1,...,x_{\ell_1+\ell_2+\ell_3}}
d_u(d_u-1)(d_u-2)\times d_v(d_v-1)(d_v-2)\times\prod_{i=1}^{\ell_1+\ell_2+\ell_3}d_{x_i}(d_{x_i}-1)
\times \frac{(|E|-1-(\ell_1+\ell_2+\ell_3+3))!}{2^{(\ell_1+\ell_2+\ell_3+3)}(|E|-1)!}.
\end{equation}

Setting $\ell=\ell_1+\ell_2+\ell_3$ and $w_{x_i}=d_{x_i}(d_{x_i}-1)$,
Proposition \ref{psum} implies that:
\[\sum_{x_1,...,x_{\ell}}\prod_{i=1}^{\ell}w_{x_i}
\times \frac{(n-\ell)!}{n!}\leq \left(\frac{\sum_{x\in V(G)} w_x}{n}\right)^{\ell}.\]
By Lemma \ref{lem:observations}(a), we have $|E|-1<n$ (since $Q<0$) and so
$\frac{(|E|-1-\ell)!}{(|E|-1)!}\times(|E|-1)^{\ell}<\frac{(n-\ell)!}{n!}\times n^{\ell}$. This implies:
\begin{equation} \label{esum}
\sum_{x_1,...,x_{\ell}}
\prod_{i=1}^{\ell}w_{x_i}
\times \frac{(|E|-1-\ell)!}{(|E|-1)!}
<\left(\frac{\sum_{x\in V(G)} w_x}{|E|-1}\right)^{\ell}
=\left(\frac{\sum_{x\in V(G)} w_x}{|E|}\right)^{\ell}\times(1+\inv{|E|-1})^{\ell}.
\end{equation}
Since $\ell\leq\inv{4} n$ and $|E|\geq\hf n$  (by Lemma \ref{lem:observations}(a)), we have $(|E|-1-(\ell+1))(|E|-1-(\ell+2))(|E|-1-(\ell+3))>\inv{10}|E|^3$ and $(1+\inv{|E|-1})^{\ell}<e^{n/(4|E|-4)}<e$.
Thus by (\ref{ew1}), the expected number of $W_1$ subgraphs is at most:
\begin{eqnarray*}
&&\frac{10e}{8|E|^3}\left(\sum_{u\in V(G)}d_u(d_u-1)(d_u-2)\right)^2
%\sum_{\ell_1,\ell_2,\ell_3\geq0}
%\left(1-\frac{\ell_1+\ell_2+\ell_3+3}{|E|}\right)^{-(\ell_1+\ell_2+\ell_3+3)}
\left(\frac{\sum_{x\in V(G)}d_x(d_x-1)}{2|E|}\right)^{\ell_1+\ell_2+\ell_3}\\
&<&\frac{5eR^2}{|E|}\sum_{\ell_1,\ell_2,\ell_3\geq0}(1+Q)^{\ell_1+\ell_2+\ell_3}\\
&=&\frac{5eR^2}{|E|}\left(\sum_{\ell\geq0}(1+Q)^{\ell}\right)^3\\
&=&\frac{5eR^2}{|E||Q|^3}<\frac{e}{20\omega(n)^3}<\frac{1}{4}\e.
\end{eqnarray*}
A nearly identical argument shows that
the expected number of subgraphs of type $W_2$ with $u\neq v$ is also at most $\frac{1}{4}\e$ -this time $\ell_1,\ell_2,\ell_3$
denote the number of vertices, other than $u,v$ on the two cycles and the path.
In the case where $u=v$, $\left(\sum_{u}d_u(d_u-1)(d_u-2)\right)^2$ is replaced with
$\sum_u d_u(d_u-1)(d_u-2)(d_u-3)$, which is smaller.

This proves that Theorem~\ref{t2}(b) holds for
a random configuration.  Proposition \ref{pcm} implies that it holds for a random graph.
\proofend

{\bf Remark:} Note that in the proof of part (b), Condition D was only used to (i) allow us
to apply part (a) to show that the size of the largest component, and hence $\ell$, is at most $\inv{4} n$, and
(ii) switch from random configurations to random graphs. Step (i) could have been carried
out without Condition D:  any bound of the form $|E|-\Theta(|E|)$ would have sufficed,
and we can obtain such a bound easily, eg. by arguing that with high probability,
there are $\theta(|E|)$ components of size 2.  Step (ii) can be carried out
under much weaker conditions than Condition D.

\section{Proof of Theorem~\ref{thm:UpperBound}(b)}\label{sub} In this section we turn to
the critical range of $Q$; i.e. $-\lambda n^{-1/3} R^{2/3}\leq Q \le \lambda n^{-1/3} R^{2/3}$.
We will bound the probability
that the size of the largest component is too big.
Without loss of generality, we can assume that $\lambda>\frac{1600}{\z}$.

Our proof follows along the same lines as that of Theorem 1 (see also Theorem 7) of \cite{np1}.

We wish to show that  there exists a constant $B>1$ such that
with probability at least $1-\epsilon$, the largest component has size at most
$B n^{2/3} R^{-1/3}$. To do so, we set $T:=n^{2/3} R^{-1/3}$ and bound the probability
that our branching process starting at a given vertex $v$ does not return to zero within $T$ steps.

Lemma \ref{lem:concQ} yields that, with high probability, $|Q_t-Q|\leq\hf |Q|+ \frac{800}{\z}n^{-1/3} R^{2/3}$
for every $t\leq T$. Since we assume $\lambda>\frac{1600}{\z}$, this implies $Q_t\leq2\lambda n^{-1/3} R^{2/3}$.

The fact that the drift, $Q_t$, may be positive makes this case a bit trickier than
that in the previous section, and so we need a more involved argument.  It will be convenient
to assume that $Y_t$ is bounded by $H:=\frac{1}{12 \lambda}n^{1/3} R^{1/3}$,
so we add $Y_t\geq H$ to our stopping time conditions.  We also need to add a condition
corresponding to the concentration of $R$. Specifically, we define
$$\gamma := \min \{t: \mbox{$(Y_t=0)$, $(Y_t \ge H)$, $(Q_{t}>2\lambda n^{-1/3} R^{2/3})$, $(|R_t-R|> R/2)$ or $(t=T)$}\}.$$
Since $\D\leq n^{1/3}R^{1/3}/\ln n$, we have $T<\frac{\z}{400}\frac{n}{\D}$ for $n$ sufficiently large.
So Lemmas \ref{lem:concR} and \ref{lem:concQ} imply that, with high probability, we will not
have $Q_{\g}>2\lambda n^{-1/3} R^{2/3}$ or $|R_{\g}-R|> R/2$.  So by upper bounding $\pr(Y_{\g}\geq H)$
and $\pr(\g=T)$, we can
obtain a good lower bound on $\pr(Y_t=0)$ which, in turn, is a lower bound on $Y_t$ reaching zero before
reaching $H$.

For $t \le \gamma$, we have $Q_{t-1} \le 2\lambda n^{-1/3} R^{2/3}$ and so:
\begin{equation}\label{ehq}
HQ_{t-1}\leq\inv{6}R
%\qquad\mbox{i.e.}\qquad Q_{t-1}\leq\inv{6}aR.
\end{equation}

For $t\leq\g$, we also have $Y_{t-1}>0$ and so $\ex(\eta_t)$ and $\ex(\eta_t^2)$ are as in (\ref{eq:Qt}) and (\ref{eq:Rt}).  We also have $R_{t-1}\geq\hf R$ and (\ref{ehq}).
For small enough $x \ge 0$,  $e^{-x} \ge 1-x+x^2/3$. So for $n$ sufficiently large,
$|\eta_t/H|\le (2+\Delta)/H<(\ln n)^{-1}$ is small enough to yield:
\begin{eqnarray*}
\Ex[e^{-\eta_t/H} | C_{t-1}] & \ge & 1-\Ex[\frac{\eta_t}{H}|C_{t-1}]
+\frac{1}{3}\Ex[\frac{\eta_t^2}{H^2}|C_{t-1}]
= 1-\frac{Q_{t-1}}{H}+\frac{R_{t-1}}{3H^2}\\
&\ge& 1-\frac{R}{6H^2}+\frac{R}{6H^2} = 1.
\end{eqnarray*}
This shows that $e^{-Y_{\min(t, \gamma)}/H}$ is a submartingale,
and  so we can apply the Optional Stopping Theorem with stopping time $\t:=\g$. As $Y_{\gamma-1} \le H$,
we have $Y_\gamma \le H+\D< 2H$. Recalling that we begin our branching process
at vertex $v$ and applying $x/4 \le 1-e^{-x}$, for $0 \le x \le 2$, we have:
$$e^{-d_v/H}=e^{-Y_0/H}\le \Ex e^{-Y_\gamma/H} \le \Ex\left[1-\frac{Y_\gamma}{4H}\right],$$
which, using the fact that for $x>0$, $1-e^{-x} \le x$, implies
\begin{equation}
\label{eq:boundYgamma}
\Ex[Y_\gamma] \le 4H(1-e^{-d_v/H}) \le 4d_v.
\end{equation}
In particular
\begin{equation}
\label{eq:boundPrH} \Pr[Y_\gamma \ge H] \le \frac{4d_v}{H}.
\end{equation}

Now we turn our attention to $\pr(\g=T)$. We begin by bounding:
\[\Ex[Y_t^2-Y_{t-1}^2|C_{t-1}] = \Ex[(\eta_t+Y_{t-1})^2-Y_{t-1}^2|C_{t-1}]=\Ex[\eta_t^2|C_{t-1}]-2\Ex[\eta_tY_{t-1}|C_{t-1}].\]
For $t\leq\g$, we have $Y_{t-1}>0$ and so $\Ex[\eta_t|C_{t-1}]=Q_{t-1}$.
Thus $\Ex[\eta_tY_{t-1}|C_{t-1}]=Q_{t-1}Y_{t-1}$.  Also, for $t \le \gamma$, we must have
$Y_{t-1}<H$, $R_{t-1}\geq\hf R$ and (\ref{ehq}) so:
\[\Ex[Y_t^2-Y_{t-1}^2|C_{t-1}]\ge R_{t-1} - 2H\max(Q_{t-1}, 0)
\ge \frac{R}{2}-\frac{R}{3} =\frac{R}{6}.\]
Thus $Y_{\min(t,\gamma)}^2 -\frac{1}{6}R\min (t,\gamma)$ is a submartingale, and so by
the Optional Stopping Theorem we have:
\[\Ex\left[ Y_\gamma^2 -\frac{R\gamma}{6}\right]\geq Y_0^2=d_v^2\geq0.\]
This, together with (\ref{eq:boundYgamma}) and the fact (derived above) that
$Y_\gamma \le 2H$, implies that
$$\Ex \gamma \le \frac{6}{R} \Ex Y_\gamma^2 \le \frac{12 H}{R} \Ex Y_\gamma \le \frac{48 H d_v}{R},$$
showing
\begin{equation}
\label{eq:PrGammaGrT} \Pr[\gamma =T] \le \frac{48 H d_v}{R T}.
\end{equation}
We conclude from (\ref{eq:boundPrH}), (\ref{eq:PrGammaGrT}), and
Lemmas~\ref{lem:concR} and~\ref{lem:concQ}  that, for $n$ sufficiently large,
\begin{eqnarray*}
\Pr[|{\cal C}_v|\geq T] &\le& \Pr[Y_\gamma>H] + \Pr[\gamma=T]+\Pr[Q_{t}>2\lambda n^{-1/3} R^{2/3}]
+\Pr[|R_{\gamma}-R|>R/2] \\
&\le& \frac{4 d_v}{H}+\frac{48 H d_v}{R T}+Tn^{-10}+Tn^{-10} \\
&\le& 48 \lambda n^{-1/3}R^{-1/3}  d_v
+\frac{48  n^{1/3} R^{1/3} d_v}{12 \lambda n^{2/3} R^{2/3}} + 2Tn^{-10}
<50 \lambda n^{-1/3} R^{-1/3}  d_v.
\end{eqnarray*}
For some constant $B\geq1$, let $N$ be the number of vertices lying in components of
size at least $K:=B n^{2/3} R^{-1/3}\ge T$. Recalling that $\sum_v d_v=2|E|<3n$ by
Lemma~\ref{lem:observations}(a), we have
\begin{eqnarray*}
\Pr[|\cm| \ge K] &\le& \Pr[N \ge K]  \le  \frac{\Ex[N]}{K}
\le \frac{1}{K} \sum_{v \in V} \Pr[{\cal C}_v \ge K] \le \frac{1}{K}
\sum_{v \in V} \Pr[{\cal C}_v \ge T] \\
&\le& \frac{1}{K} \sum_{v \in V}
50 \lambda n^{-1/3} R^{-1/3}  d_v=\frac{50\lambda}{nB}\sum_{v}d_v
< \frac{150 \lambda}{B},
\end{eqnarray*}
which can be made to be less than $\e$ by taking $B$ to be sufficiently
large.
This proves that Theorem~\ref{thm:UpperBound}(b) holds for
a random configuration.  Proposition \ref{pcm} implies that it holds for a random graph.
\proofend

\section{Proof of Theorem~\ref{thm:UpperBound}(a)}\label{stua}
In this section we bound the probability
that the size of the largest component is too small when $Q$ is in the critical range.
Our proof follows along the same lines as that of Theorem 2 of \cite{np1}.

Recall that we have $-\lambda n^{2/3} R^{2/3}\leq Q \le \lambda n^{-1/3} R^{2/3}$.
Without loss of generality, we can assume that $\lambda>\frac{1600}{\z}$.

We wish to show that  there exists a constant $A>0$ such that
with probability at least $1-\e$, the largest component has size at least
$A n^{2/3} R^{-1/3}$.

We will first show that, with sufficiently high
probability, our branching process reaches a certain value $h$. Then we will show
that, with sufficiently high probability, it will take at least
$An^{2/3} R^{-1/3}$ steps for it to get from $h$ to zero, and thus there
must be a component of that size.

We set $T_1:=n^{2/3} R^{-1/3}$ and $T_2:=An^{2/3} R^{-1/3}$.
For $t \le T_1+T_2 \le 2 n^{2/3} R^{-1/3}$ (for $A\leq 1$),
Lemma \ref{lem:concQ}  yields that, with high probability, $|Q_t-Q|\leq\hf |Q|+ \frac{800}{\z}n^{-1/3} R^{2/3}$ and thus
(since $\lambda>\frac{1600}{\z}$)
%\begin{equation}
\[Q_t\ge -2 \lambda n^{-1/3} R^{2/3}.\]
%\end{equation}
%
We set
\[h:=
%\frac{A^{1/4}R}{8}\times \frac{1}{2 \lambda n^{-1/3} R^{2/3}}=
A^{1/4} n^{1/3}R^{1/3}\]
%so that
%
%\begin{equation}
%h \max\{|Q|+|x|: x \in {\cal I}\} \le \frac{A^{1/4} R}{8} \le
%\frac{R}{8},
%\end{equation}
%
so that if $Q_t\ge -2 \lambda n^{-1/3} R^{2/3}$ and $A<(16\lambda)^{-4}$ then
\begin{equation}
\label{eq:assignh} h Q_t \ge -2\lambda A^{1/4} R \geq -\frac{R}{8}.
\end{equation}

We start by showing that $Y_t$ reaches $h$, with sufficiently high probability.
To do so, we define $\t_1$ analogously to $\g$ from
Section \ref{sub}, the only difference being that we allow $Y_t$ to return to zero before $t=\t_1$.
$$\tau_1= \min \{t : \mbox{$(Y_t \ge h)$,  $(Q_t< -2 \lambda n^{-1/3} R^{2/3})$, $(|R_t-R| > R/2)$, or $(t=T_1)$}\}.$$
We wish to show that, with sufficiently high probability, we get $Y_{\t_1}\geq h$.  We know that
the probability of $Q_{\t_1}< -2 \lambda n^{-1/3} R^{2/3}$ or $|R_{\t_1}-R| > R/2$ is small by Lemmas~\ref{lem:concR} and~\ref{lem:concQ}.  So it remains to bound $\pr(\t_1=T_1)$.
For $t \le \tau_1$, if $Y_{t-1}>0$, then by (\ref{eq:Qt}), (\ref{eq:Rt}), (\ref{eq:assignh}) and the fact that $Y_{t-1}<h$:
$$
\Ex[Y_t^2-Y_{t-1}^2|C_{t-1}] = \Ex [\eta_{t}^2|C_{t-1}] +
2\Ex[\eta_t Y_{t-1}|C_{t-1}] \ge R_{t-1}+2h \min(Q_{t-1},0) \ge
\frac{R}{2} - \frac{R}{4} \ge R/4,
$$
Also if $Y_{t-1}=0$, then by (\ref{eq:etat2FromZero}) we have
$$
\Ex[Y_t^2-Y_{t-1}^2|C_{t-1}]=\Ex [\eta_{t}^2|C_{t-1}] \ge
R_{t-1}/2\ge R/4.
$$
Thus $Y_{\min(t,\tau_1)}^2 - \frac{1}{4}R\min(t, \tau_1)$ is a submartingale, so we can apply
the Optional Stopping Theorem to obtain:
$$\Ex Y_{\tau_1}^2-\frac{R}{4} \Ex \tau_1 \ge Y_0^2 \ge 0,$$
and as $Y_{\tau_1} \le 2h$,
\[\Ex \t_1\leq\frac{4}{R}\Ex Y_{\tau_1}^2\leq\frac{16 h^2}{R}.\]
Hence
\begin{equation}
\Pr[\tau_1 = T_1] \le \frac{16 h^2}{R T_1}.
\end{equation}
By the bound $\D\leq n^{1/3}R^{1/3}/\ln n$, we have $T_1+T_2<\frac{\z}{400}\frac{n}{\D}$.
So Lemmas~\ref{lem:concR}~and~\ref{lem:concQ} imply that for
sufficiently large $n$,
\begin{equation}
\label{eq:boundSuccessI} \Pr[Y_{\tau_1} < h] \le \Pr[\tau_1=T_1] +
\Pr[Q_{\t_1}< -2 \lambda n^{-1/3} R^{2/3}]+\Pr[|R_{\tau_1}-R| > R/2]\le
\frac{16h^2}{R T_1}+ 2 T_1n^{-10} < 20\sqrt{A}.
\end{equation}

This shows that with probability at least $1-20\sqrt{A}$, $Y_t$ will reach $h$
within $T_1$ steps.  If it does reach $h$, then the largest component must have size at least $h$,
which is not as big as we require.  We will next show that, with sufficiently high probability,
it takes at least $T_2$ steps for $Y_t$ to return to zero, hence establishing that the
component being exposed has size at least $T_2$, which is big enough to prove the theorem.
Define
$$\tau_2=\min\{s : \mbox{$(Y_{\tau_1+s}=0)$,
$(Q_{\t_1+s}< -2 \lambda n^{-1/3} R^{2/3})$, $(|R_{\tau_1+s}-R| > R/2)$, or $(s=T_2)$} \}.$$
We wish to show that, with sufficiently high probability, we get $\t_2=T_2$ as this implies
$Y_{\t_1+T_2-1}>0$. We know that the probability of $Q_{\t_1+\t_2}< -2 \lambda n^{-1/3} R^{2/3}$ or
$|R_{\t_1+\t_2}-R| > R/2$ is small by Lemmas~\ref{lem:concR} and~\ref{lem:concQ}.  So it remains to bound $\pr[Y_{\tau_1+s}=0]$.

It will be convenient to view the random walk back to $Y_t=0$ as a walk from 0 to $h$ rather
than from $h$ to 0; and it will also be convenient if that walk never drops below 0.
So we define $M_s=h-\min\{h,Y_{\tau_1+s}\}$, and thus $M_s\geq0$ and $M_s=h$ iff $Y_{\t_1+s}=0$.
If $0<M_{s-1}<h$, then $M_{s-1}=h-Y_{\tau_1+s-1}$ and since $M_s \le |h-Y_{\tau_1+s}|$, we have in this case that:
\begin{eqnarray}
\nonumber M_s^2-M_{s-1}^2 &\le& (h-Y_{\tau_1+s})^2-(h-Y_{\tau_1+s-1})^2\\
\nonumber &=&2h(Y_{\tau_1+s-1}-Y_{\tau_1+s})+Y_{\tau_1+s}^2-Y_{\tau_1+s-1}^2\\
\nonumber &=&\eta_{\tau_1+s}(Y_{\tau_1+s}+Y_{\tau_1+s-1}-2h)\\
\nonumber &=&\eta_{\tau_1+s}(\eta_{\tau_1+s}-2M_{s-1})\\
\label{eq:increaseInMs1}&=&\eta_{\tau_1+s}^2 - 2 \eta_{\tau_1+s} M_{s-1}.
\end{eqnarray}
If $M_{s-1}=0$, then $Y_{\tau_1+s-1}\geq h$ and so
\begin{equation}
\label{eq:increaseInMs2}
M_s^2-M_{s-1}^2=M_s^2 \le \eta_{\tau_1+s}^2.
\end{equation}
For $1 \le s \le \tau_2$, we have $M_{s-1}<h$, (\ref{ehq})
and by (\ref{eq:etat2FromZero}) we have $\ex[\eta^2_{\t_1+s}]=R_{\t_1+s}\leq\frac{3}{2}R$
since $|R_{\t_1+s}-R|\le R/2$. Applying those, along with
(\ref{eq:increaseInMs1}), (\ref{eq:increaseInMs2}) and
(\ref{eq:assignh}) we conclude that such values of $s$,
\begin{eqnarray*}
\Ex[M_s^2-M_{s-1}^2|C_{\tau_1+s-1},\tau_1]  &\le&
\max(\Ex[\eta_{\tau_1+s}^2|C_{\tau_1+s-1},\tau_1],\Ex[\eta_{\tau_1+s}^2
- 2 \eta_{\tau_1+s} M_{s-1}|C_{\tau_1+s-1},\tau_1]) \\
&\le& \max\left(\frac{3R}{2},\frac{3R}{2}-2hQ_{\tau_1+s-1}\right) \le \frac{3R}{2}+\frac{R}{4} < 2R.
\end{eqnarray*}
Let $\Ex_h$ and $\Pr_h$ denote  respectively the conditional
expectation and the conditional probability given the event
$\{Y_{\tau_1} \ge h\}$. So $M_{s\wedge \tau_2}^2 - 2R(s \wedge
\tau_2)$ is a supermartingale under $\Ex_h$, and the Optional
Stopping Theorem yields:
$$\Ex_h[M_{\tau_2}^2 - 2R \tau_2] \le \Ex_h M_0^2=0.$$
This, along with the fact that $\t_2\leq T_2$ yields:
$$\Ex_h M_{\tau_2}^2 \le 2R \Ex_h \tau_2 \le 2T_2 R.$$
Hence by (\ref{eq:boundSuccessI}) and
Lemmas~\ref{lem:concR}~and~\ref{lem:concQ}, we have that for $n$ sufficiently large:
\begin{eqnarray*}
\Pr_h[\tau_2 < T_2] &\le& \Pr_h[M_{\tau_2} \ge h]
+\Pr_h[Q_{\t_1+\t_2}< -2 \lambda n^{-1/3} R^{2/3}]
+\Pr_h[|R_{\tau_1+\tau_2}-R|>R/2] \\ &\le& \frac{\Ex_h
M_{\tau_2}^2}{h^2} + \frac{2T_2n^{-10}}{\Pr[Y_{\tau_1} \ge h]}  \le
\frac{2 T_2 R}{h^2}+ \frac{2T_2n^{-10}}{1-20\sqrt{A}} \le \frac{3T_2 R}{h^2}.
\end{eqnarray*}
Combining this with (\ref{eq:boundSuccessI}) we conclude
$$\Pr[|\cm| < T_2] \le \Pr[\tau_2 < T_2] \le \Pr[Y_{\tau_1} < h]+ \Pr_h[\tau_2 < T_2] \le  20\sqrt{A} + \frac{3T_2R}{h^2}
=23\sqrt{A}<\e,$$
for $A<(\frac{\e}{23})^2$. (Recall that we also require $A<(16\lambda)^{-4}$.)
This proves that Theorem~\ref{thm:UpperBound}(a) holds for
a random configuration.  Proposition \ref{pcm} implies that it holds for a random graph.
\proofend

{\bf Remark:}  Recall that in Section \ref{sd} we said that if one vertex $v$ has degree $\D\gg n^{1/3}$
and all other vertices have small degrees - small enough that the degree sequence obtained by removing $v$
has $R=O(1)$ and $Q=O(n^{-1/3})$ - then with high probability there will be a component of size $O(n^{1/3})$.  To prove this,
we follow the proof of Theorem \ref{thm:UpperBound}(a),
beginning the branching process with $v$.
Note that this yields $R_i=O(1)$ for every $i\geq 0$, and this allows us to replace $R$ by $R_0=O(1)$
throughout the proof.  Thus, eg. we set $T_1:=n^{2/3} R_0^{-1/3}, T_2:=An^{2/3}R_0^{-1/3}$ and
$h:=A^{1/4}n^{1/3}R_0^{1/3}$. Lemmas~\ref{lem:concR}~and~\ref{lem:concQ}
are easily seen to hold with $R$ replaced by $R_0$.
Note that (for $n$ sufficiently large) $h<\D=d_v=Y_0$, and so we can skip the first part of the proof,
where we show that $Y_i$ eventually reaches $h$ with high probability.

\section{Proof of Theorem \ref{t3}}
We close this paper with the supercritical range; i.e. when $Q>\omega(n) n^{-1/3}R^{2/3}$,
where $\omega(n)$ grows with $n$.  We wish to show that there exists a constant $A$  such that
with probability at least $1-\e$, the largest component has size at least $AQn/R$.

The same argument as used for the proof of Theorem~\ref{thm:UpperBound}(a) applies here.
In fact, the argument is a bit simpler here since we will always have the drift $Q_t>0$.

We fix $A$ later, and set $h=A^{1/4}\sqrt{Qn}$, $T_1=\frac{\z}{2000}Qn/R$ and $T_2=AQn/R$.
If $A<\frac{\z}{2000}$ then Lemmas \ref{lem:concR}~and~\ref{lem:concQ} imply
that for any $t\leq T_1+T_2$ we have, with high probability, $|R_t-R|\leq\hf R$ and
$|Q_t-Q|\leq \hf Q+ \frac{800}{\z}n^{-1/3} R^{2/3}<\frac{3}{4}Q$ for $n$ sufficiently large in terms of $\omega$.
So we define our stopping times as:
\begin{eqnarray*}
\tau_1&=& \min \{t : \mbox{$(Y_t \ge h)$,  $(Q_t< \inv{4}Q)$, $(|R_t-R| > R/2)$, or $(t=T_1)$},\}\\
\tau_2&=&\min\{s : \mbox{$(Y_{\tau_1+s}=0)$,
$(Q_{\t_1+s}<\inv{4}Q)$, $(|R_{\tau_1+s}-R| > R/2)$, or $(s=T_2)$}.
\end{eqnarray*}
Note that the analogue (\ref{ehq}) holds trivially since for $t\leq\t_1+\t_2$ we have $Q_t>0$.
In fact, (\ref{ehq}) was only required to deal with the possibility that $Q_t$ was negative,
and so it is not needed for this case.

For $n$ sufficiently large, the same arguments (simplified slightly since $Q_t\geq0$) still yield:
\begin{eqnarray*}
\Pr[Y_{\t_1}<h]&\leq&\frac{16h^2}{R T_1}+ 2 T_1n^{-10} < 20\sqrt{A}\times\frac{2000}{\z},\\
\Pr_h[\tau_2 < T_2] &\le&  \frac{3T_2 R}{h^2} \leq 3\sqrt{A},
\end{eqnarray*}
and so $\Pr[|\cm| < T_2] \le 20\sqrt{A}\times\frac{2000}{\z} + 3\sqrt{A}<\e$ for $A$ sufficiently small.
This proves that Theorem~\ref{t3} holds for
a random configuration.  Proposition \ref{pcm} implies that it holds for a random graph.
\proofend


\begin{thebibliography}{99}
\bibitem{acl}
W. Aiello, F. Chung and L. Lu. {\em A random graph model for massive graphs.} Proceedings of the Thirty-Second Annual ACM Symposium on Theory of Computing (2000), 171--180.
\bibitem{abt} R. Arratia, A. Barbour and S. Tavar\'e. {\em Logarithmic Combinatorial Structures:
a Probabilistic Approach.} European Math Society, Zurich (2003).
\bibitem{bc} E. Bender and R. Canfield. {\em The asymptotic number of labelled graphs with given degree sequences.}
J. Comb. Th. (A) {\bf 24} (1978), 296~-~307.
\bibitem{bb} B. Bollob\'as. {\em A probabilistic proof of an asymptotic formula for the number of labelled graphs.}
Europ. J. Comb. {\bf 1} (1980), 311~-~316.
\bibitem{bb2} B. Bollob\'as. {\em The evolution of random graphs.} Trans. Am. Math. Soc. {\bf 286} (1984),
257~-~274.
\bibitem{bbrg} B. Bollob\'as. {\em Random Graphs.} 2nd Edition. Cambridge University Press (2001).
\bibitem{cl} F. Chung and L. Lu. {\em Concentration inequalities and martingale inequalities: a survey.}
Internet Math. {\bf 3} (2006), 79~-~127.
\bibitem{cf} C. Cooper and A. Frieze. {\em The size of the largest strongly connected component
of a random digraph with a given degree sequence.} Combin. Prob. \& Comp. {\bf 13} (2004), 319~-~338.
\bibitem{er} P. Erd\H{o}s and A. R\'enyi. {\em On the evolution of random graphs.}
Magayar Tud. Akad. Mat. Kutato Int. Kozl. {\bf 5} (1960), 17~-~61.
\bibitem{fr} N. Fountoulakis and B. Reed. {\em Critical conditions for the emergence of a giant component.}
Proceedings of Eurocomb 2009. (Journal version in preparation.)
\bibitem{sj} S. Jansen. {\em The probability that a random multigraph is simple.} Combin. Prob. \& Comp.
{\bf 18} (2009), 205~-~225.
\bibitem{jl} S. Jansen and M. {\L}uczak. {\em A new approach to the giant component problem.}
Rand. Struc \& Alg. (to appear).
\bibitem{ks} M. Kang and T. Seierstad. {\em The critical phase for random graphs with a given degree sequence.}
Combin. Prob. \& Comp {\bf 17} (2008), 67~-~86.
\bibitem{rk} R. Karp. {\em The transitive closure of a random digraph.} Rand. Struc \& Alg.
{\bf 1} (1990), 73~-~94.
\bibitem{ypbook} D. Levin, Y. Peres and E. Wilmer. {\em Markov Chains and Mixing Times.}
American Math. Soc. (2008).
\bibitem{tl} T. {\L}uczak. {\em Component behaviour near the critical point of the random graph process.}
Rand. Struc. \& Alg. {\bf 1} (1990), 287~-~310.
%\bibitem{mw} B. McKay and N. Wormald. {\em Asymptotic enumeration by degree sequence of graphs
%with degrees $o(n^{1/2})$.} Combinatorica {\bf 11} (1991), 369~-~382.
\bibitem{mr1} M. Molloy and B. Reed. {\em A critical point for random graphs with a given degree
sequence.} Rand. Struc. \& Alg. {\bf 6} (1995), 161~-~180.
\bibitem{mr2} M. Molloy and B. Reed. {\em The size of the largest component of a random graph
on a fixed degree sequence.} Combin. Prob. \& Comp {\bf 7} (1998), 295~-~306.
\bibitem{np1} A. Nachmias and Y. Peres. {\em The critical random graph with martingales.}
Israel J. Math. (to appear).
\bibitem{np2} A. Nachmias and Y. Peres. {\em Component sizes of the random graph outside the
scaling window.} Latin Am. J. Prob. and Math. Stat. {\bf 3} (2007), 133~-~142.
\bibitem{nbw} M. Newman, A. Barab\'asi and D. Watts. {\em The Structure and Dynamics of Networks.}
Princeton University Press (2006).

\end{thebibliography}
\end{document}